\documentclass [a4paper,11pt]{amsart}
\usepackage{graphicx}
\usepackage[all]{xy}

\usepackage{amssymb,amscd}
\usepackage{mathrsfs}

\newtheorem{thm}{Theorem}[section]

\def\C{\mathbb C}

\def\dim{\operatorname{dim}}

\def\id{\operatorname{id}}

\usepackage{graphicx}
\usepackage[ansinew]{inputenc}    % accents i altres
\usepackage[all]{xy}
\usepackage{amssymb,amscd}
\usepackage{color}

\newtheorem{cor}[thm]{Corollary}
\newtheorem{teo}[thm]{Theorem}
\newtheorem{lem}[thm]{Lemma}
\newtheorem{prop}[thm]{Proposition}

\theoremstyle{definition}
\newtheorem{ex}[thm]{Example}
\newtheorem{defi}[thm]{Definition}

\newtheorem{remark}[thm]{Remark}

\def\C{\mathbb C}

\def\dim{\operatorname{dim}}

\def\id{\operatorname{id}}

\def\Derlog{\operatorname{Derlog}}

\begin{document}
\title{Equisingularity of map germs from a surface to the plane}

\author{J. J. Nuño-Ballesteros, B. Oréfice-Okamoto, J. N. Tomazella}

\address{Departament de Geometria i Topologia,
Universitat de Val\`encia, Campus de Burjassot, 46100 Burjassot
SPAIN}

\email{Juan.Nuno@uv.es}

\address{Departamento de Matem\'atica, Universidade Federal de S\~ao Carlos, Caixa Postal 676,
13560-905, S\~ao Carlos, SP, BRAZIL}

\email{bruna@dm.ufscar.br}

\address{Departamento de Matem\'atica, Universidade Federal de S\~ao Carlos, Caixa Postal 676,
13560-905, S\~ao Carlos, SP, BRAZIL}

\email{tomazella@dm.ufscar.br}

\thanks{The first author is partially supported by DGICYT Grant MTM2012--33073 and CAPES-PVE Grant 88881.062217/2014-01. The second author is partially supported by FAPESP Grant 2013/14014-3. The third author is partially supported by CNPq Grant 309626/2014-5 and
FAPESP Grant 2013/10856-0.}


\begin{abstract}
Let $(X,0)$ be an ICIS of dimension 2 and let $f:(X,0)\to (\C^2,0)$ be a map germ with an isolated instability. We look at the invariants that appear when $X_s$ is a smoothing of $(X,0)$ and $f_s:X_s\to B_\epsilon$ is a stabilization of $f$. We find relations between these invariants and also give necessary and sufficient conditions for a $1$-parameter family to be Whitney equisingular. As an application, we show that a family $(X_t,0)$ is Zariski equisingular if and only if it is Whitney equisingular and the numbers of cusps and double folds of a generic linear projection are constant with respect to $t$. 
\end{abstract}

\maketitle

\section{Introduction}

Let $(X,0)\subset (\C^n,0)$ be an isolated complete intersection singularity (ICIS) of dimension 2 and let $f:(X,0)\to (\C^2,0)$ be a map germ with an isolated instability. Consider $X_s$ a smoothing of $(X,0)$ and $f_s:X_s\to B_\epsilon$ a stabilization of $f$, with $s\in D$,  where $D$ is an open neighbourhood of $0$ in $\C$ and $B_\epsilon$ is a ball of radius $\epsilon$ in $\C^2$. This means that for $s=0$, $X_0=X$ and $f_0=f$ and for $s\ne0$, $X_s$ is smooth and $f_s$ is (locally) stable. From the celebrated paper of Whitney \cite{W}, we know that the only singularities of $f_s$ are either folds, cusps or double folds. The numbers of cusps and double folds in $f_s$, for $s$ small enough, do not depend on the stabilization and we denote them by $c$ and $d$, respectively. We also denote by $S$ the singular locus of $f$ and by $\Delta:=f(S)$ the discriminant.

In the case when $X:=\C^2$, Rieger \cite{Rieger} and Gaffney and Mond \cite{GM} showed that
$$\mu(\Delta)=\mu(S)+2c+2d,\quad
c=\mu(S)+m-2,
$$
where $\mu$ denotes the Milnor number and $m$ is the degree of $f$. One first consequence is that $c,d$ are topological invariants of $f$. Another important consequence is that they control the Whitney equisingularity of a $1$-parameter family $f_t:(\C^2,0)\to(\C^2,0)$. In fact, we have the following equivalent statements (see \cite{G,GM}):
\begin{enumerate}
\item $f_t$ is Whitney equisingular,
\item $f_t$ is topologically trivial,
\item $\mu(\Delta_t)$ is constant on $t$,
\item $c_t,d_t$ are both constant on $t$.
\end{enumerate}

The main goal in this paper is to extend all the above results for the case when $(X,0)$ is a 2-dimensional ICIS. On one hand, we generalize both formulas. The first one is exactly the same but in the second one, we have to add $\mu(X,0)$ due to the contribution of the singular point in $(X,0)$. Moreover, we give a characterization of the Whitney equisingularity of a $1$-parameter family $f_t:(X_t,0)\to(\C^2,0)$ in terms of the constancy of the numerical invariants.

We also consider the particular case that the map $f$ is a generic linear projection $p:(X,0)\to(\C^2,0)$. Then, the numbers $c,d$ do not depend on the projection $p$ and turn out to be invariants of $(X,0)$. As an application of the formulas, we deduce that a 1-parameter family $(X_t,0)$ is Zariski equisingular (for a generic linear projection) if and only if it is Whitney equisingular and $c_t,d_t$ are constant on $t$. The results of this section are related to those of Henry and Merle \cite{HM}, although their viewpoint is slightly different, since they consider the varieties of cusps and double folds of the total space of the deformation $(\mathcal X,0)$. Other results about classification of singularities of projections of complete intersections can be found in \cite{Go}.

In the last part of the paper, we assume that $(X,0)$ and $f$ are both weighted homogeneous (with the same weights) and  we give explicit formulas for all the invariants in terms of the weights and degrees. Our formulas extend the results of Gaffney and Mond in the smooth case \cite{GM2}.

\section{Maps on ICIS}

We recall in this section the basic definitions and results for maps on ICIS, following Mond and Montaldi \cite{MM}.
Let $(X,0)\subset(\C^n,0)$ be an ICIS and let $f:(X,0)\to(\C^p,0)$ be an analytic map germ. The singular locus of $f$, $S$, is equal to the union of the singular set of $(X,0)$ and the set of the points where $X$ is smooth and $f$ is not a submersion. We say that $f$ has \emph{finite singularity type} if the restriction $f|_{S}$ is finite.

Let $f,g:(X,0)\to(\C^p,0)$ be analytic map germs. We say that $f$ and $g$ are \emph{$\mathscr A$-equivalent} if there exist biholomorphism germs $\omega:(X,0)\to(X,0)$ and $\xi:(\C^p,0)\to(\C^p,0)$ such that the following diagram is commutative:
$$
\xymatrix{(X,0)\ar[r]^-f\ar[d]^\omega&(\C^p,0)\ar[d]^-\xi\\
(X,0)\ar[r]^-g&(\C^p,0)}
$$

An \emph{unfolding} of $f:(X,0)\to(\C^p,0)$ is a germ $F:(\mathcal X,0)\to(\C^r\times\C^p,0)$ together with $\pi:(\mathcal X,0)\to(\C^r,0)$, a flat deformation of $(X,0)$, such that the diagram
$$
\xymatrix{(\mathcal X,0)\ar[d]^-{\pi}\ar[r]^-{F}&(\C^r\times\C^p,0)\ar[ld]^-{p_1}\\
(\C^r,0)&}
$$
commutes and $p_2\circ F|_X=f$, where $p_1$ and $p_2$ are the projection to the first and to the second coordinate, respectively.

For simplicity, we assume that $(\mathcal X,0)\subset(\C^r\times\C^n,0)$ and that $\pi(u,x)=u$. Then $F(u,x)=(u,f_u(x))$ and $f_0=f$. For each $u\in\C^r$, we write $X_u=\pi^{-1}(u)$ and $f_u=(F|_{X_u}):X_u\to\C^p$.

We say that two unfoldings $F:(\mathcal X,0)\to(\C^r\times\C^p,0)$ and $G:(\tilde{\mathcal X},0)\to(\C^r\times\C^p,0)$ are \emph{$\mathscr A$-equivalent} (as unfoldings) if there exist biholomorphism germs $\Phi:(\mathcal X,0)\to(\tilde{\mathcal X},0)$ and $\Psi:(\C^r\times\C^p,0)\to(\C^r\times\C^p,0)$ such that
$$
\xymatrix{(\mathcal X,0)\ar[r]^-F\ar[d]^-\Phi&(\C^r\times\C^p,0)\ar[d]^\Psi\\
(\tilde{\mathcal X},0)\ar[r]^-G&(\C^r\times\C^p,0)}
$$
commutes and moreover $\Phi$ and $\Psi$ are unfoldings of their respective identity maps.

We say that $F$ is \emph{$\mathscr A$-trivial} if it is $\mathscr A$-equivalent to the constant unfolding $$\xymatrix{(\C^r\times X,0)\ar[r]^-{\mbox{id}\times f}&(\C^r\times\C^p,0)}.$$ 
The germ $f:(X,0)\to(\C^p,0)$ is called \emph{stable} if any unfolding of it is $\mathscr A$-trivial.

Given $F:(\mathcal X,0)\to (\C^r\times\C^p,0)$ an unfolding of $f$ and $h:(\C^s,0)\to(\C^r,0)$ an analytic map germ. The pull-back of $F$ by $h$ is the unfolding $h^*F$ defined by the diagram
$$
\xymatrix{(\C^s\times_{\C^r}\mathcal X,0)\ar[d]^-{\pi'}\ar[r]^-{h^*F}&(\C^s\times\C^p,0)\ar[ld]^-{p_1}\\
(\C^s,0)&}
$$
where $\C^s\times_{\C^r}\mathcal X$ is the fibered product, that is,
$$\C^s\times_{\C^r}\mathcal X=\{(v,(u,x))\in\C^s\times \mathcal X: h(v)=u\},
$$ 
and $\pi'(v,(u,x))=v$, $h^*F(v,(u,x))=(v,f_{h(v)}(x))$.

 We say that $F$ is a \emph{versal} unfolding if any unfolding of $f$ is $\mathscr A$-equivalent to $h^*F$ for some map germ $h$.

\medskip
Like in the case of maps $f:(\C^n,0)\to (\C^p,0)$, there exist infinitesimal criteria to decide if an unfolding of a map $f:(X,0)\to(\C^p,0)$ is a versal unfolding (see \cite{MM}). We denote by $\Theta_n$ the $\mathscr O_n$-module of the vector fields in $(\C^n,0)$ (where $\mathscr O_n$ is the ring of the analytic function germs from $(\C^n,0)$ to $\C$), and by $\Derlog(X,0)$ the $\mathscr{O}_n$-submodule of $\Theta_n$ consisting of those vector fields tangent to $X$. Moreover, we denote by $\Theta_{X,0}$  the $\mathscr{O}_{X,0}$-module of vector fields on $(X,0)$ (where $\mathscr O_{X,0}$ is the ring of the analytic function germs from $(X,0)$ to $\C$). We have:
$$\Theta_{X,0}\cong \frac{\Derlog(X,0)}{I(X,0)\Theta_n}, $$
where $I(X,0)$ is the ideal of $\mathscr O_n$ of functions vanishing on $X$.

\begin{defi}
The \emph{$\mathscr A_e$-codimension} of $f$ is defined as
$$\mathscr A_e\mbox{-codim}(f)=\dim_\C\frac{\Theta(f)}{tf(\Theta_{X,0})+\omega f(\Theta_p)},$$
where $\Theta(f)$ is the set of germs at 0 of vector fields along $f$,
$$\begin{matrix}tf:&\Theta_{X,0}&\rightarrow&\Theta(f)\\&\eta&\mapsto&df\circ \eta\end{matrix}$$
and
$$\begin{matrix}\omega f:&\Theta_p&\rightarrow&\Theta(f)\\&\varsigma&\mapsto&\varsigma\circ f\end{matrix}$$ 
We say that $f$ is \emph{$\mathscr A$-finite} if this codimension is finite.
\end{defi}

\begin{thm}\label{estabilizacao}(\cite{MM}) Let $(X,0)$ be an ICIS and $f:(X,0)\to(\C^p,0)$ be a map germ with finite singularity type. The minimal number of parameters in a versal unfolding of $f$ is equal to $\mathscr A_e$-$\mbox{codim}(f)+\tau(X,0),$
where $\tau(X,0)$ is the Tjurina number of $(X,0)$, that is, the minimal number of parameters in a versal deformation of $(X,0)$.
\end{thm}

As a corollary of this theorem, we have the following consequences. The second one is a generalization of the Mather-Gaffney criterion (see \cite{wall}).

\begin{cor}\label{infinitesimal} Let $(X,0)$ be an ICIS and let $f:(X,0)\to(\C^p,0)$ be an analytic map germ with finite singularity type.
\begin{enumerate}
\item $f$ is stable if and only if $X$ is smooth and $f$ is stable in the usual sense.
\item $f$ is $\mathscr A$-finite if and only if $f$ has isolated instability (i.e., there is a representative $f:X\to V$ such that for any $y\in V\setminus\{0\}$, the multigerm of $f$ at $f^{-1}(y)\cap S$ is stable).
\end{enumerate}
\end{cor}

Another consequence of the theorem is the existence of stabilizations, at least in the range of Mather's nice dimensions. Given an analytic map germ $f:(X,0)\to(\C^p,0)$, a \emph{stabilization} is a 1-parameter unfolding $F:(\mathcal X,0)\to (\C\times\C^p)$ with the property that for all $s\ne0$ small enough, the map $f_s:X_s\to B_\epsilon$ is stable, where $B_\epsilon$ is a ball of radius $\epsilon$ in $\C^p$.
By Theorem \ref{estabilizacao}, if $f:(X,0)\to(\C^p,0)$ is $\mathscr A$-finite, a stabilization of $f$ exists if $(d,p)$ are nice dimensions in the Mather's sense ($d=\dim(X,0)$). Moreover, the topological type of $f_s:X_s\to B_\epsilon$ is independent of $s$ and of the stabilization (see \cite{MM}).
 
Given a stabilization of $f$, we denote by $\Delta_s:=f_s(S_s)$ the discriminant of $f_s:X_s\to B_\epsilon$. By \cite[Theorem 2.1]{MM}, $\Delta_s$ has the homotopy type of a wedge of spheres with dimension $p-1$. The number of spheres in this wedge is called the discriminant Milnor number of $f$ and is denoted $\mu_\Delta(f)$. We also find in \cite{MM} the following remarkable theorem.

\begin{thm}\label{conjecture}(\cite{MM}) Let $(X,0)$ be an ICIS of dimension $d$ and let $f:(X,0)\to(\C^p,0)$ be $\mathscr A$-finite, with $d\ge p$ and $(d,p)$ nice dimensions. Then,
$$
\mathscr A_e\mbox{-codim}(f)+\tau(X,0)\le \mu_\Delta(f),
$$
with equality if $X$ and $f$ are weighted homogeneous (with the same weights).
\end{thm}

\section{Maps from a surface to the plane} 

From now on, we will restrict ourselves to maps
$$
f:(X,0)\to(\C^2,0)
$$
where $(X,0)=(\phi^{-1}(0),0)\subset (\C^n,0)$ is an ICIS of dimension $2$. 

The singular locus of $f$ is
\begin{equation*}
S=v(\left\langle \phi\right\rangle+\det(J(\phi,f))),
\end{equation*}
where $J(\phi,f)$ is the Jacobian matrix of the map $(\phi,f)$, $\det$ denotes the determinant of the matrix and $v(I)$ denotes the zero locus of the ideal $I$ in $\mathscr O_n$. We will consider $S$ with the analytic structure given by this ideal.
The discriminant of $f$ is the variety $\Delta=f(S)$, where $\Delta$ is considered with the analytic structure defined by the 0th Fitting ideal (in the sense of \cite{T,MP}) of the restriction map $f|_{S}:S\to (\C^2,0)$. We remark that, since the 0th Fitting ideal is a principal ideal, $\Delta$ is a hypersurface.

\begin{lem} 
Let $(X,0)$ be an ICIS of dimension $2$ and let $f:(X,0)\to(\C^2,0)$ be a germ with finite singularity type. The following statements are equivalent:
\begin{enumerate}
  \item $f$ is $\mathscr A$-finite.
  \item $\Delta$ is a reduced curve.
  \item $S$ is a reduced curve and $f:S\to\Delta$ is one-to-one.
\end{enumerate}
\end{lem}
\begin{proof}

Assume $f$ is $\mathscr A$-finite. By the second part of Corollary \ref{infinitesimal}, $f$ has an isolated instability. Moreover, by the first part, for all the points $x\in X-\{0\}$, $f$ is stable in the usual sense and therefore, the only singularities of $\Delta$ are cusps and nodes, which are isolated, so $\Delta$ is reduced.

Conversely, assume $\Delta$ is reduced, then the only singular point of $\Delta$ is the origin, in a small enough representative. Since $(f|_{S})^{-1}(0)=\{0\}$, $f$ is stable outside the origin and hence, $f$ is $\mathscr A$-finite by Corollary \ref{infinitesimal}.

The statements (2) and (3) are equivalent by \cite[Lemma 2.3]{MNP}.

\end{proof}

If $f:(X,0)\to(\C^2,0)$ is $\mathscr A$-finite, then $S$ and $\Delta$ are reduced curves, so it makes sense to consider their Milnor numbers $\mu(S)$ and $\mu(\Delta)$ ($\Delta$ is a plane curve and for $S$ we take the definition of Buchweitz-Greuel \cite{BG}). Another curve which is relevant to understand the geometry of $f$ is $f^{-1}(\Delta)$. Again this is a reduced curve and we consider its Milnor number $\mu(f^{-1}(\Delta))$. Other important geometrical invariants are the numbers $c,d$, which are defined as follows.

\begin{defi}\label{defi-cd}
Let $f:(X,0)\to(\C^2,0)$ be $\mathscr A$-finite, where $(X,0)$ is a 2-dimensional ICIS.
Given $F:(\mathcal X,0)\to(\C\times\C^2,0)$ a stabilization of $f$, we fix a small representative
$F:\mathcal X\to D\times B_\epsilon$, where $D$ is an open neighbourhood of $0$ in $\C$ and $B_\epsilon$ is a ball of radius $\epsilon$ in $\C^2$.
We define:
\begin{itemize}
\item $c=$ number of cusps of $f_s:X_s\to B_\epsilon$,
\item $d=$ number of double folds of $f_s:X_s\to B_\epsilon$,
\end{itemize}
for $s\in D-\{0\}$. We will deduce from Proposition \ref{cusps} and Theorem \ref{1} that these numbers are independent of $s$ and of the stabilization.
\end{defi}
%We remark that $c(f)$ and $d(f)$ could be defined if $\dim X>2$.
      
Inspired by the work of Gaffney and Mond (\cite{GM}), we will show how to compute these invariants algebraically and we present formulas relating them. Before doing it, we remark that the surfaces $S(F)$, $\Delta(F)$ and $F^{-1}(\Delta(F))$ are flat deformations of $S$, $\Delta$ and $f^{-1}(\Delta)$, respectively. In fact, $\Delta(F)$ is a hypersurface in $(\C^3,0)$ and $S(F)$, $F^{-1}(\Delta(F))$ are complete intersections in $(\C^{n+1},0)$. Thus, all of them are Cohen-Macaulay and the projection onto the parameter space is flat. Another important fact is that $S(F)$ is a smoothing of $S$, since $S_s$ is smooth if $f_s$ is stable.

%Gaffney and Mond showed in \cite{GM} that, in the case when $X=\C^2$, the aforementioned numbers relate in the following way:
%
%\begin{equation*}
%\mu(\Delta(f))=\mu(S)+2c(f)+2d(f)
%\end{equation*}
%and
%\begin{equation*}
%c(f)=\mu(S)+m(f)-2.
%\end{equation*}

In \cite{GM}, Gaffney and Mond show how to compute the number of cusps $c$ as the dimension of an algebra. In our case, we have the following result.

\begin{prop}\label{cusps}
Let $\phi:(\C^n,0)\to (\C^{n-2},0)$ be an analytic map germ such that $(X,0)=(\phi^{-1}(0),0)$ is an ICIS. Let $f:(X,0)\to(\C^2,0)$ be an $\mathscr A$-finite map germ.
Then,
\begin{equation*}
c=\dim_{\C}\frac{\mathscr O_n}{\left\langle \phi\right\rangle+I_n(J(f,\phi,\delta))},
\end{equation*}
where $\delta=\det(J(f,\phi))$ and $I_n(J(f,\phi,\delta))$ is the ideal generated by the minors of size $n$ of $J(f,\phi,\delta)$.
\end{prop}

\begin{proof}
Let $F:(\mathcal X,0)\to(\C\times\C^2,0)$ be a stabilization of $f$ like in the previous section. Let $\Phi:(\C\times\C^n,0)\to(\C^{n-2},0)$ be a deformation of $\phi$ so that $\Phi^{-1}(0)=(\mathcal X,0)$. Put $M_F=J(f_s,\phi_s,\delta_s)$, where $\delta_s=\det(J(f_s,\phi_s))$ (the Jacobian considered just in the variables $x$). Note that $\delta_s$ is one of the $n$-minors of $M_F$, so the set
\begin{equation*}
(Y,0)=v(I_n(M_F)+\left<\Phi\right>)\subset (\C^{n+1},0)
\end{equation*}
is given by the pairs $(s,x)$ such that either $(s,x)=(0,0)$ or $s\ne0$ and $f_s$ has a cusp at $x\in X_s$. Therefore, $\dim (Y,0)=1$ and, since
$$
\dim \mathscr O_{\mathcal X,0}-[(n+1)-n+1](n-n+1)=1=\dim\frac{\mathscr O_{\mathcal X,0}}{I_n(M_F)},
$$  
we conclude, by the results of Eagon-Hochster (\cite{EH}), that $(Y,0)$ is a determinantal variety and hence, Cohen-Macaulay.

We consider, now, the map
$$\begin{matrix}\pi:&(Y,0)&\to&(\C,0)\\
&(s,x)&\mapsto&s\end{matrix}.$$
For $s\ne0$ small enough, since $(Y,0)$ is Cohen-Macaulay,  we have:
\begin{align*}
c=\#\pi^{-1}(s)&=e(\left<s\right>,\frac{\mathscr O_{\mathcal X,0}}{I_n(M_F)})=\dim_\C\frac{\mathscr O_{\mathcal X,0}}{I_n(M_F)+\left<s\right>}\\&=\dim_\C\frac{\mathscr O_{X,0}}{I_n(J(f,\phi,\delta))},
\end{align*}
where $e(I,R)$ denotes the Samuel multiplicity of an ideal of definition $I$ in a Noetherian local ring $R$.
\end{proof}

%\begin{ex}
%Let $(X,0)\subset(\C^3,0)$ be the ICIS defined by the zero locus of the map $\phi(x,y,z)=x^2+y^2-z^2$ and let $f:(X,0)\to(\C^2,0)$ be the map germ defined by $f(x,y,z)=(x^3+z^2,y^2+z^2)$. In this case, $c(f)=18$.
%\end{ex}

Now, the number of double folds $d$ can be calculated by the following theorem, as it was done in \cite{GM,Rieger} for maps $f:(\C^2,0)\to(\C^2,0)$.

\begin{teo}\label{1}
Let $(X,0)$ be an ICIS of dimension $2$ and let $f:(X,0)\to(\C^2,0)$ be an $\mathscr A$-finite map germ.
Then,
\begin{equation*}
c+d=\frac{1}{2}\left[\mu(\Delta)-\mu(S)\right].
\end{equation*}
\end{teo}
\begin{proof}
Let $F:(\mathcal X,0)\to(\C^2,0)$ be a stabilization of $f$ and take a representative as in Definition \ref{defi-cd}. Since the singular set of $F$, $S(F)$ is a Cohen-Macaulay variety and satisfies Serre's R1 condition, we have that $S$ is normal and, hence, $F|_{S(F)}:S(F)\to \Delta(F)$ is a normalization of $\Delta(F)$. Therefore, from \cite[Page 607]{T} we have
\begin{equation*}
\delta(\Delta_s)=\delta(\Delta)-\delta(S),
\end{equation*} 
where $\delta(.)$ denotes the delta invariant of the curve and $\Delta_s$ is the discriminant of $f_s$. We remark that the left-hand side is a sum of local contributions from the singularities of
$\Delta_s$, which are cusps (corresponding to cusps and each contributing 1) and nodes (correponding to double folds and each contributing 1), then
\begin{equation*}
c+d=\delta(\Delta)-\delta(S).
\end{equation*} 
From the Milnor formula, we know that $\delta(\Delta)=\frac{1}{2}(\mu(\Delta)+r(\Delta)-1)$ and $\delta(S)=\frac{1}{2}(\mu(S)+r(S)-1)$, where $r(.)$ denotes the number of branches of the curve. 
However, since the map $f:S\to\Delta$ is one-to-one, we have that $r(\Delta)=r(S)$. Thus,
\begin{equation*}
c+d=\frac{1}{2}[\mu(\Delta)-\mu(S)].
\end{equation*}
\end{proof}

\begin{remark}
As it was done in \cite{GM} for the smooth case,  if $(X,0)$ is an ICIS of dimension $2$ and $f:(X,0)\to(\C^2,0)$ is an $\mathscr A$-finite map germ, then
\begin{equation*}
c+d=\dim_\C\frac{\mathscr O_2}{\mathcal F_1(f|_{S})},
\end{equation*}
where $\mathcal F_1(f|_{S})$ is the first Fitting ideal of $\mathscr O_S$ as $\mathscr O_2$-module via $f$. This is a consequence of \cite[Theorem 3.6]{MP}.
\end{remark}

In the case when $X:=\C^2$, that is, if we consider a map germ $f:(\C^2,0)\to(\C^2,0)$ with an isolated instability, Gaffney and Mond, in \cite{GM}, show that
\begin{equation}\label{gm2}
c=\mu(S)+m-2,
\end{equation}
where $m$ is the degree of $f$. In order to generalize this formulas, we will need the following lemma, which is a generalization of the well known Riemann-Hurwitz formula.

\begin{lem} Let $X$ and $Y$ be analytic sets and $f:X\to Y$ be a map.  We denote by $B$ the branch locus of $f$ and $Z=f^{-1}(B)$. Assume that $f:X-f^{-1}(B)\to Y-B$ is a degree $m$ covering. Then
\begin{equation}\label{RH}
\chi(X)=m\chi(Y)+\chi(Z)-m\chi(B),
\end{equation}
where $\chi(\cdot)$ denotes the Euler characteristic.
\end{lem}

\begin{proof} 
Since the map $f|_{X-Z}:X-Z\to Y-B$ is a covering map with $m$ sheets,
\begin{equation*}\chi(X-Z)=m\chi(Y-B).\end{equation*}

Furthermore, since the real codimension of $Z$ in $X$ and the real codimension of $B$ in $Y$ are even, we have that $\chi(X-Z)=\chi(X)-\chi(Z)$ and $\chi(Y-B)=\chi(Y)-\chi(B)$.
Therefore,
\begin{eqnarray*}
\chi(X)&=&\chi(X-Z)+\chi(Z)\\
&=&m\chi(Y-B)+\chi(Z)\\
&=&m\chi(Y)-m\chi(B)+\chi(Z).\end{eqnarray*}
\end{proof}

\begin{thm}\label{2}
If $(X,0)$ is a $2$-dimensional ICIS and $f:(X,0)\to(\C^2,0)$ is $\mathscr A$-finite, then
\begin{equation*}
c+\mu(X,0)=\mu(S)+m-2,
\end{equation*}
where $m$ is the degree of $f$.
\end{thm}

\begin{proof}
Let $F:(\mathcal X,0)\to(\C^2,0)$ be a stabilization of $f$ and take a representative as in Definition \ref{defi-cd}.
First, we apply the formula (\ref{RH}) to each map
$$f_s:f_s^{-1}(\Delta_s)\to \Delta_s.$$
In this case, $B$ is the set of the cusps and the nodes, we have
\begin{eqnarray*}
\chi(f_s^{-1}(\Delta_s))&=&(m-1)\chi(\Delta_s)+\chi(f^{-1}(B))-(m-1)\chi(B)\\
&=&(m-1)\chi(\Delta_s)+(m-2)(c+d)-(m-1)(c+d)\\
&=&(m-1)\chi(\Delta_s)-c-d,
\end{eqnarray*}
where the second equality follows from the fact each cusp has $m-3$ regular and $1$ singular inverse images and each node has $m-4$ regular and $1$ singular inverse images.

Next, we apply the formula \ref{RH} to each map
$f_s:X_s\to B_\epsilon.$ In this case, the branch locus is $\Delta_s$.
We have
\begin{equation*}
\chi(X_s)=m\chi(B_\epsilon)+\chi(f_s^{-1}(\Delta_s))-m\chi(\Delta_s).
\end{equation*}
Hence,
\begin{eqnarray*}
\mu(X,0)+1&=&\chi(X_s)\\
&=&m+((m-1)\chi(\Delta_s)-c-d)-m\chi(\Delta_s)\\
&=&m-\chi(\Delta_s)-c-d.
%&=&m(f)-1+\mu(\Delta(f))-3c(f)-2d(f)\\
%&=&m(f)-1+(\mu(S)+2c(f)+2d(f))-3c(f)-2d(f)\\
%&=&m(f)-1+\mu(S)-c(f).
\end{eqnarray*}

Moreover, since $\Delta$ is a plane curve, from \cite{BG},
\begin{eqnarray*}
\chi(\Delta_s)&=&1-\mu(\Delta)+\mu(\Delta_s)\\
&=&1-\mu(\Delta)+2c+d\\
&=&1-\mu(S)-d,
\end{eqnarray*}
where $\mu(\Delta_s)$ is global Milnor number (that is, the sum of all the Milnor number of the singular point of $\Delta_s$).
Therefore,
\begin{eqnarray*}
\mu(X,0)+1&=&m-(1-\mu(S)-d)-c-d\\
&=&m-1+\mu(S)-c,
\end{eqnarray*}
or equivalently,
\begin{equation*}
c+\mu(X,0)=\mu(S)+m-2.
\end{equation*}
\end{proof}

We recall that the discriminant Milnor number defined by Mond and Montaldi (\cite{MM}), $\mu_\Delta(f)$ is, by definition, equal to $1- \chi(\Delta_s)$. Therefore, as a corollary of the proof of the previous theorem, we can relate this number to the other invariants.

\begin{cor}\label{mudelta}
If $(X,0)$ is a $2$-dimensional ICIS and $f:(X,0)\to(\C^2,0)$ is $\mathscr A$-finite, then
\begin{eqnarray*}
\mu_\Delta(f)=&\mu(\Delta)-2c-d\\=&d+\mu(S).
\end{eqnarray*}
\end{cor}

We use the ideas of \cite[Lemma 3.2]{NT} to prove the following result.

\begin{lem}\label{32}
Let $(X,0)$ be an ICIS of dimension $2$ and let $f:(X,0)\to(\C^2,0)$ be an $\mathscr A$-finite map germ, then 
\begin{equation*}
(m-1)\mu(\Delta)=\mu(f^{-1}(\Delta))+m-2,
\end{equation*}
where $f$ has degree $m$.
\end{lem}
\begin{proof}
We have
\begin{equation*}
\mu(\Delta)-\mu(\Delta_s)=1-\chi(\Delta_s)
\end{equation*}
and
\begin{equation*}
\mu(f^{-1}(\Delta))-\mu(f_s^{-1}(\Delta_s))=1-\chi(f_s^{-1}(\Delta_s)).
\end{equation*}

From the proof of Theorem \ref{2}, we know that $\chi(f_s^{-1}(\Delta_s)=(m-1)\chi(\Delta_s)-(c+d)$.
%We denote by $\Sigma_s$ the set of cusps and transverse double folds of $f_s$ and $\Delta(f_s)^0=\Delta(f_s)-\Sigma_s$. We have, then,
%\begin{equation*}
%\chi(\Delta(f_s))=\chi(\Delta(f_s)^0)+c(f)+d(f)
%\end{equation*}
%and, since each cusp or transverse double fold has $m(f)-2$ inverse image, we have
%\begin{equation*}
%\chi(f_s^{-1}(\Delta(f_s))=\chi(f_s^{-1}(\Delta(f_s)^0))+(m(f)-2)(c(f)+d(f))
%\end{equation*}
%
%Furthermore, $f_s:f^{-1}(\Delta(f_s)^0)\to\Delta(f_s)^0$ is a covering of $m(f)-1$ sheets, therefore,
%\begin{equation*}
%\chi(f_s^{-1}(\Delta(f_s)^0))=(m(f)-1)\chi(\Delta(f_s)^0)
%\end{equation*}
%and then
%\begin{eqnarray*}
%\chi(f_s^{-1}(\Delta(f_s)))&=&(m(f)-1)(\chi(\Delta(f_s))-c(f)-d(f))+(m(f)-2)(c(f)+d(f))\\
%&=&(m(f)-1)(\chi(\Delta(f_s)))-(c(f)+d(f))
%\end{eqnarray*}
On the other hand, since the singularities of $\Delta_s$ are cusps and nodes (corresponding to the cusps and transverse double-fold of $f_s$, respectively) and the Milnor number of these singularities are $2$ and $1$, respectively, we have
\begin{equation*}
\mu(\Delta_s)=2c+d.
\end{equation*}
Moreover, the singularities of $f_s^{-1}(\Delta_s)$ can be either a simple cusp, if it is a regular inverse image of a cusp, or a singularity of type $A_3$ if it is a singular inverse image of a cusp, or a node, if it is an inverse image of a transverse double point (both regular or singular). The Milnor numbers of these type of singularities are $2$, $3$ or $1$ respectively and each cusp has $m - 3$ regular inverse images and $1$ singular inverse image, and each transverse double point has $m - 2$ inverse images, therefore
\begin{equation*}
\mu(f_s^{-1}(\Delta_s))=2(m-3)c+3c+(m-2)d.
\end{equation*}
And then,
\begin{eqnarray*}
\mu(f^{-1}(\Delta))&=&\mu(f_s^{-1}(\Delta_s))+1-\chi(f_s^{-1}(\Delta_s))\\
&=&2(m-3)c+3c+(m-2)d+1-\\&&((m-1)(\chi(\Delta_s))-(c+d))\\
&=&2(m-3)c+3c+(m-2)d+1-\\&&((m-1)(\mu(\Delta_s)-\mu(\Delta)+1)-(c+d))\\
&=&2(m-3)c+3c+(m-2)d+1-\\&&((m-1)((2c+d)-\mu(\Delta)+1)-(c+d))\\
&=&(m-1)\mu(\Delta)+2-m
\end{eqnarray*}
Therefore,
\begin{equation*}
(m-1)\mu(\Delta)=\mu(f^{-1}(\Delta))+m-2.
\end{equation*}

\end{proof}

Finally, we also will need the following lemma, whose proof is inspired in the proof of \cite[Lemma 2.3]{MNP}.

\begin{lem}\label{deltasmooth} Let $(X,0)$ be an ICIS and let $f:(X,0)\to(\C^2,0)$ be a map germ with finite singularity type. Let $y\in\Delta$ such that $\Delta$ is smooth at $y$, then $ f^{-1}(y)\cap S$ consists of a single point. Furthermore, if   $ f^{-1}(y)\cap S=\{x\}$, then $X$ is smooth at $x$ and $f$ has fold type at $x$.
\end{lem}
\begin{proof}
We take representatives $S$ of $(S,0)$ and $V$ of $(\C^2,0)$ such that $f|_{S}:S\to V$ is a finite morphism of complex spaces. Denote by $\mathcal F_0(f|_{S})$ the 0th Fitting ideal sheaf of $f|_S$. We assume that $v(\mathcal F_0(f|_{S}))$ is smooth at $y$ and we write $f|_{S}^{-1}(y)=\{x_1,\dots,x_r\}$. Then, the stalk at $y$ is the product
\begin{equation*}
\mathcal F_0({f|_{S})}_y=F_0({f|_{S}}_{x_1})\dots F_0({f|_{S}}_{x_r}),
\end{equation*}
where ${f|_{S}}_{x_i}:(S,x_i)\to (V,y)$ is the germ of $f|_{S}$ at $x_i$. Since each $F_0({f|_{S}}_{x_i})\in \mathcal M _{V,y}$ (the maximal ideal of $\mathscr O_{V,y}$), we must have $r=1$.

We write $f|_{S}^{-1}(y)=\{x\}$. Let $q=\dim_\C\mathscr O_{S,x}/f^*\mathcal M_{V,y}$. By \cite{MP}, a minimal presentation of $\mathscr O_{S,x}$ has the form
$$
\begin{CD}
\mathscr O_{V,y}^q@>\lambda>>\mathscr O_{V,y}^q @>\alpha>> \mathscr O_{S,x}@>>> 0.
\end{CD}
$$
Then ${\mathcal F_0(f|_{S})}_y$ is generated by $\det(\lambda)\in \mathcal M_{V,y}^q$ so that $q$ is necessarily equal to $1$. The exactness of the sequence implies $F_0(f|_{S})_y=Im(\lambda)=Ker(\alpha)$ and thus $\mathscr O_{S,x}$ is isomorphic to $\mathscr O_{2,y}/\mathcal F_0(f|_{S})_y$, which is regular. On the other hand, since $q=1$ we have $f^*(\mathcal M_{V,y})=\mathcal M_{S,x}$ and $f|_{S}$ is regular. But this is implies that $X$ is smooth at $x$ and $f$ has fold type at $x$.

%Therefore, in order to show that $f$ is fold type at $x$, we need to show that $X$ is regular at $x$. We have,
%\begin{equation*}
%S=v(\left\langle \phi_1,\dots,\phi_p\right\rangle+I_{p+2}(J(f_1,f_2,\phi_1,\dots,\phi_p))).
%\end{equation*}

%We write $g_1,\dots,g_r$ the generators of $I_{p+2}(J(f_1,f_2,\phi_1,\dots,\phi_p))$. %Let $Y=v(g_1,\dots,g_r)$. We have
%$$\codim Y\leq(p+2-(p+2)+1)(n-(p+2)+1=n-p-1,$$
%therefore $\rank d(g_1,\dots,g_r)(x)\leq \codim Y\leq n-p-1$ and $n-1=\rank d(\phi,g_1,\dots,g_r)(x)\leq \rank d\phi(x)+\rank d(g_1,\dots,g_r)(x)\leq \rank d\phi(x)+n-p-1$.
%Hence, $\rank d\phi(x)\geq p$, that is, $X$ is not singular at $x$.

%Hence from the proof of \cite[Lemma 2.3]{MNP}, $\#(f^{-1}(y))=1$ and if $f^{-1}(y)=\{x\}$, then $S$ is smooth in $x$ and $f|_{S}$ is not singular in $x$. Therefore, $f$ is a fold type germ on $x$. 
\end{proof}

%\begin{cor} Let $(X,0)\subset(\C^n,0)$ be an ICIS and let $f:(X,0)\to(\C^2,0)$ be a map germ  with finite singularity type. The following sentences are equivalent.
%\begin{enumerate}
 % \item $f$ is $\mathscr A$-finite.
  %\item $\mu(\Delta(f),0)<\infty$.
  %\item $c(f)<\infty$ and $d(f)<\infty$.
%\end{enumerate}
%\end{cor}
%\begin{proof}
%By the the lemma \ref{reduzido}, the claim (1) is equivalent to the claim (2).

%$(2)\Leftrightarrow(3)$ follows from the theorems \ref{1} e \ref{2}.

%$(2)\Rightarrow (3)$: If $\mu(\Delta(f))<\infty$, then  $c(f)<\infty$ and $d(f)<\infty$ by the Theorem \ref{1}.

%$(3)\Rightarrow (2)$: If $c(f)<\infty$, by the Theorem \ref{2}, $\mu(S)<\infty$. Hence, by the Theorem \ref{1}, $\mu(\Delta(f))<\infty$.

%$(2)\Rightarrow (1)$: From the Lemma \ref{deltasmooth}, the only point of $S$ which is not a fold type is $\{0\}$. Therefore, $\mathscr A_e-\codim(X,f)<\infty$.

%$(1)\Rightarrow (2)$: If $\mathscr A_e-\codim(X,f)<\infty$, then there exists a neighbourhood of the origin in $\Delta(f)$ where all the points except $0$ are cusps or double folds. Since these types of singularities are isolated, we have that $\Delta(f)$ has an isolated singularity at the origin and, hence, $\mu(\Delta(f),0)<\infty$.
%\end{proof}

\section{Equisingularity of maps}

In this section, we consider 1-parameter families of maps $f_t:(X_t,0)\to(\C^2,0)$. We want to characterize the Whitney equisingularity of the family in terms of the constancy of the invariants. We first recall the basic definitions about equisingularity of varieties and of maps.

Let $(X,0)\subset(\C^n,0)$ be an ICIS and let $\pi:(\mathcal X,0)\to (\C,0)$ be a 1-parameter flat deformation, where $(\mathcal X,0)\subset(\C\times\C^n,0)$ and $\pi(t,x)=t$. We assume that $0\in X_t$ for each $t$, hence we can consider the family of ICIS $\{(X_t,0)\}$. 

\begin{defi}\label{variedadewhitney} With the above notation:
\begin{enumerate}
  \item We say that $(\mathcal X,0)$ is \emph{good} if there is a representative $\mathcal X$ defined in $D\times U$, where $D$ and $U$ are open neighbourhoods of the origin in $\C$ and $\C^n$, respectively, such that $X_t \backslash \{0\}$ is smooth for any $t\in D$.
%  \item We say that $(\mathcal X,0)$ is \emph{$\mu$-constant} if $\mu(X_t,0)$ is constant.
  \item We say that $(\mathcal X,0)$ is \emph{topologically trivial} if there is a homeomorphism germ $\Phi:(\mathcal X,0)\to(\C\times X,0)$ such that $\pi\circ\Phi=\pi$.
  \item We say that $(\mathcal X,0)$ is \emph{Whitney equisingular} if it is good and the pair $(\mathcal X\setminus T,T)$ satisfy the Whitney conditions, where $T=D\times\{0\}$.
\end{enumerate}
\end{defi}

The condition for the family to be good means that the ICIS $X_t$ have an isolated singularity ``uniformly'', that is, in a neighbourhood $U$ which does not depend on $t$. A topologically trivial family is not good in general. Moreover, according to the first isotopy lemma of Thom, any Whitney equisingular family is topologically trivial, but the converse is not true in general. 

One of the problems in equisingularity theory is to find numerical invariants, the constancy of which in the family implies the Whitney equisingularity. In the case of families of ICIS $(\mathcal X,0)$ the answer to this problem is given in \cite{G}: a family of ICIS of dimension $d$, $(\mathcal X,0)$, is Whitney equisingular if and only if the polar multiplicities $m_0(X_t,0), \dots, m_{d}(X_t,0)$ are constant. In particular,  if $d=2$, we have the following relation:
$$
m_0(X_t,0)+m_2(X_t,0)=\mu(X_t,0)+m_1(X_t,0)+1.
$$
Since all the invariants in this formula are upper semicontinuous, we deduce that $(\mathcal X,0)$ is Whitney equisingular if and only if $\mu(X_t,0)$ and $m_1(X_t,0)$ are constant. 

\medskip
Let $f:(X,0)\to (\C^2,0)$ be an $\mathscr A$-finite map germ on a 2-dimensional ICIS $(X,0)$. We consider a 1-parameter unfolding $F:(\mathcal X,0)\to (\C\times\C^2,0)$ such that $(\mathcal X,0)$ is a deformation of $(X,0)$ as above and $F$ is defined by $F(t,x)=(t,f_t(x))$. We assume that $F$ is origin preserving, that is, $f_t(0)=0$ for any $t$. Thus, we have a family of function germs $\{f_t:(X_t,0)\to(\C^2,0)\}$.

\begin{defi} 
We say that $F$ is \emph{good} if there exists a representative $F:\mathcal X\to D\times V$, where $D$ is an open neighbourhood of the origin in $\C$ and $V=B_\epsilon$ is a ball of radius $\epsilon$ in $\C^2$,  such that $X_t \backslash \{0\}$ is smooth, $f_t^{-1}(0)=\{0\}$ and $f_t$ is stable on $X_t\backslash\{0\}$ for any $t\in D$.

We say that $F$ is \emph{excellent} if it is good and in addition, $f_t$ has no cusps nor double folds on $X_t\backslash\{0\}$, for any $t\in D$.
\end{defi}

If $F$ is excellent, then we can consider the stratification $(\mathcal A,\mathcal A')$ defined as follows:
\begin{align*}
\mathcal A&=\{\mathcal X-F^{-1}(\Delta(F)), F^{-1}(\Delta(F))-T,T\},\\
\mathcal A'&=\{D\times V-\Delta(F), \Delta(F)-T', T'\},
\end{align*}
where $T=D\times\{0\}\subset\mathcal X$
and $T'=D\times\{0\}\subset D\times V$.

\begin{defi}
We say that $F$ is \emph{Whitney equisingular} if it is excellent and the above stratification $(\mathcal A,\mathcal A')$ is regular (i.e., $\mathcal A$, $\mathcal A'$ are Whitney stratifications and we have also the Thom $A_f$ condition, see \cite{GW} for details).

We say that $F$ is \emph{topologically trivial} if there are homeomorphism map germs $\Phi:(\mathcal X,0)\to(\C\times X,0),$ and $\Psi:(\C\times\C^2,0)\to(\C\times\C^2,0)$ which are unfoldings of the identity maps and 
such that $F=\Psi\circ (\id\times f)\circ\Phi$.
\end{defi}

\begin{lem}\label{44}
Let $(X,0)$ be an ICIS of dimension $2$ and let $f:(X,0)\to (\C^2,0)$ be $\mathscr A$-finite. Let $F:(\mathcal X,0)\to(\C\times\C^2,0)$ be an unfolding of $f$. If $\mu(\Delta_t)$ is constant, then $F$ is excellent, where $\Delta_t$ denotes the discriminant of $f_t$.
\end{lem}
\begin{proof}
Since $\mu(\Delta_t)$ is constant, we have that $\Delta_t$ is good (see \cite{NT}). Therefore $\Delta_t-\{0\}$ is smooth and, by Lemma \ref{deltasmooth}, $f_t$ is stable and has not cusps nor double folds on $X_t-\{0\}$.
\end{proof}

\begin{prop}\label{54} Let $(X,0)\subset(\C^n,0)$ be an ICIS of dimension $2$ and let $f:(X,0)\to(\C^2,0)$ be an $\mathscr A$-finite map germ. Let $F:(\mathcal X,0)\to(\C\times\C^2)$ be an unfolding of $f$ such that $(\mathcal X,0)$ is a good topologically trivial unfolding of $(X,0)$. Then
$\mu(\Delta_t)$ is constant if and only if $c_t$ and $d_t$ are constant.
Moreover, if $F$ is topologically trivial, then $\mu(\Delta_t)$ is constant.
\end{prop}
\begin{proof} If $\mu(\Delta_t)$ is constant, then by Theorem \ref{1}, $\mu(S_t)$, $c_t$, $d_t$ are also constant, since all these invariants are upper semicontinuous. Assume now that $c_t,d_t$ are constant. Since $(\mathcal X,0)$ is good and topologically trivial, $\mu(X_t,0)$ is constant (see \cite[Corollary 4.4]{NOT2}). Then,  Theorem \ref{2} implies that $\mu(S_t)$ and $m_t$ are also constant. Finally, again by Theorem \ref{1}, we have the constancy of $\mu(\Delta_t)$.

To see the second part, we observe that if $F$ is topologically trivial, the homeomorphism in the target preserves the discriminant $\Delta_t$. Thus, $\Delta_t$ is also topologically trivial and hence, $\mu(\Delta_t)$ is constant.
\end{proof}

Again, by using the appropriate version of Thom's second isotopy lemma for complex analytic maps, it follows that any Whitney equisingular unfolding is topologically trivial. In the following theorem, we exhibit necessary and sufficient conditions for an unfolding to be Whitney equisingular.

\begin{thm}\label{we}
Let $F:(\mathcal X,0)\to(\C\times\C^2,0)$ be an unfolding of an $\mathscr A$-finite germ $f:(X,0)\to(\C^2,0)$, where $(X,0)$ is an ICIS of dimension $2$. Then, $F$ is Whitney equisingular if and only if $\mu(X_t,0)$, $m_1(X_t,0)$, $\mu(\Delta_t)$ and $m_0(f_t^{-1}(\Delta_t))$ are constant.
\end{thm}

\begin{proof}
If $F$ is Whitney equisingular, then $(\mathcal X,0)$ is Whitney equisingular and hence $\mu(X_t,0)$, $m_1(X_t,0)$ are constant. Moreover $F$ is topologically trivial, so $\mu(\Delta_t)$ is constant by Proposition \ref{54}. Finally, the pair $(F^{-1}(\Delta(F))-T,T)$ satisfy the Whitney conditions, which implies that $m_0(f_t^{-1}(\Delta_t))$ has to be also constant.
 
Conversely, assume now that $\mu(X_t,0)$, $m_1(X_t,0)$, $\mu(\Delta_t)$ and $m_0(f_t^{-1}(\Delta_t))$ are constant. By Lemma \ref{44}, $F$ is excellent and it makes sense to consider the stratification $(\mathcal A,\mathcal A')$.
Since $\Delta_t$ is a plane curve, we have that $m_0(\Delta_t)$ is also constant. Hence, $F:F^{-1}(\Delta_F)\to\Delta_F$ is Whitney equisingular by \cite[Lemma 2.8]{NT}. After shrinking the representatives if necessary, we get that the stratification 
$(\mathcal B,\mathcal B')$ is regular, where $\mathcal B= \{F^{-1}(\Delta(F))-T,T\}$ and $\mathcal B' = \{\Delta(F)-T', T'\}$.

On the other hand, $\mathcal X$ is Whitney equisingular, so the pair $(\mathcal X-T,T)$ satisfy the Whitney conditions. But $\mathcal X-F^{-1}(\Delta(F))\subset \mathcal X-T$, thus $(\mathcal X-F^{-1}(\Delta(F)),T)$ also satisfy the Whitney conditions. The Whitney conditions for the pair $(\mathcal X-F^{-1}(\Delta(F)),F^{-1}(\Delta(F))-T)$ follow from the fact that $F$ is excellent. Hence, $\mathcal A$ is a Whitney stratification.

Analogously, $\mathcal A'$ is also a Whitney stratification, since we are only adding the open stratum $D\times V-\Delta(F)$ to $\mathcal B'$. Furthermore, since $F$ is regular in each strata, $(\mathcal A,\mathcal A')$ satisfies the Thom $A_f$ condition. Therefore $(\mathcal A,\mathcal A')$ is regular.
\end{proof}

\section{Application: Zariski equisingularity of ICIS}

Let $(X,0)\subset (\C^n,0)$ be an ICIS of dimension $2$. We denote by $L(\C^n,\C^2)$ the set of the linear projections from $\C^n$ to $\C^2$ and consider the 
family of linear projections $G:X\times L(\C^n,\C^2)\to\C^2$ defined by $G(x,p)=p(x)$. For each $p\in L(\C^n,\C^2)$, we write $G_p(x)=p(x)$, that is, $G_p$ is the restriction $p|_X:(X,0)\to (\C^2,0)$.

By the results of \cite{T}, we have that $\mu(\Delta(G_p))$ does not depend on the choice of a generic linear projection $p\in L(\C^n,\C^2)$. Note that $\Delta(G_p)$ is nothing but the projection of the polar curve $S(G_p)$ of $(X,0)$. By Theorem \ref{1}, we have:
\begin{equation*}
\mu(\Delta(G_p))=\mu(S(G_p))+2c(G_p)+2d(G_p),
\end{equation*}
hence, $\mu(S(G_p))$, $c(G_p)$ and $d(G_p)$ do not depend on the generic linear projection $p$. 

\begin{defi} We define the \emph{number of cusps} and \emph{double folds} of $(X,0)$ as $c(X,0)=c(p|_X)$ and $d(X,0)=d(p|_X)$ respectively, where $p|_X:(X,0)\to(\C^2,0)$ is the restriction of a generic linear projection $p$ in the sense of Teissier \cite{T}.
\end{defi}

 %Therefore, we have the formulas $\mu(\Delta(p))-\mu(S(p))=2(c(p)+d(p))$ and $c(p)+\mu(X,0)=\mu(S(p),0)+m(p)-2$. In this case, $S(p)$ is the polar curve of $(X,0)$, $\Delta(p)$ is the projection of the polar curve, $m(p)$ is the multiplicity of $(X,0)$ and, therefore, $c(p)$ and $d(p)$ are invariants which only on $(X,0)$.

In the following examples, we compute the invariants $c(X,0)$ and $d(X,0)$ with the aid of Singular. In order to compute the discriminant, we use the algorithm for Fitting ideals developed by Hernandes, Pe\~nafort and Miranda in \cite{HMP}

\begin{ex}
Let $(X,0)\subset (\C^3,0)$ be the surface $x^3+y^3+z^4=0$ and let $p(x,y,z)=(x+y-z,2x-y-z)$. We have, $\mu(X,0)=12$, $\mu(S(p|_X))=19$, $\mu(\Delta(p|_X))=35$, $c(X,0)=8$ and therefore, $d(X,0)=0$.
\end{ex}

\begin{ex}
Let $(X,0)\subset (\C^4,0)$ be the surface given by $x^2+y^2+z^2+w^2=yz+w^2+x^3=0$ and take the projection $p(x,y,z,w)=(x+y-z+w,2x-y-z)$. In this case we get $\mu(X,0)=7$, $\mu(S(p|_X))=17$, $\mu(\Delta(p|_X))=49$, $c(X,0)=12$ and hence, $d(X,0)=4$.
\end{ex}

We recall now the definition of Zariski equisingularity \cite{Z}, although we use the more restrictive version of Speder \cite{S}, which requires that the projection is generic so that it implies the Whitney conditions. For simplicity, we state the definition only in the case of surfaces, although it can be easily adapted inductively for higher dimensions.

Let $(X,0)\subset(\C^n,0)$ be a 2-dimensional ICIS and let $\pi:(\mathcal X,0)\to (\C,0)$ be a good deformation in the sense of Definition \ref{variedadewhitney}, with $(\mathcal X,0)\subset(\C\times\C^n,0)$ and $\pi(t,x)=t$. We also denote $T=\C\times\{0\}$.

\begin{defi} We say that $(\mathcal X,0)$ is \emph{Zariski equisingular} if there exists a generic linear projection  $P:\C\times\C^n\to\C^3$ such that:
\begin{enumerate}
\item $\ker P\cap T=\{0\}$,
\item the discriminant $\Delta(P)$ of the restriction $P:(\mathcal X,0)\to(\C^3,0)$ is 
Whitney equisingular along $P(T)$.
\end{enumerate}
\end{defi}

By condition (1), we can choose linear coordinates in $\C^3$ such that $P$ is written in the form $P(t,x)=(t,p_t(x))$, so that for each $t\in\C$ we have a generic linear projection $p_t:(X_t,0)\to (\C^2,0)$. Thus, the projection $\pi:\Delta(P)\to(\C,0)$ given by $\pi(t,z)=t$ defines a deformation of plane curves, whose fibres are precisely the discriminants $\Delta(p_t)$.
As we said before, Speder showed in \cite{S} that Zariski equisingularity implies Whitney equisingularity. We show in next theorem the relation between these two concepts for 2-dimensional ICIS.

\begin{thm}
Let $(\mathcal X,0)$ be a deformation of an ICIS of dimension $2$. The following statements are equivalent:
\begin{enumerate}
  \item $(\mathcal X,0)$ is Zariski equisingular;
  \item $(\mathcal X,0)$ is Whitney equisingular and $c(X_t,0)$ and $d(X_t,0)$ are constant.
\end{enumerate}
\end{thm}

\begin{proof}
Assume that $(\mathcal X,0)$ is Zariski equisingular, then it is Whitney equisingular by the Speder's results. Furthermore, for a generic linear projection $P(t,x)=(t,p_t(x))$, the family $\Delta(P)$ is Whitney equisingular and thus, $\mu(\Delta(p_t))$ is constant. Hence, $c(X_t,0)$ and $d(X_t,0)$ are constant by Proposition \ref{54}.

Conversely, if $(\mathcal X,0)$ is Whitney equisingular and $c(X_t,0)$ and $d(X_t,0)$ are constant, then, by Proposition \ref{54}, $\mu(\Delta(p_t))$ is constant for a generic linear projection $P(t,x)=(t,p_t(x))$. Since $\Delta(p_t)$ are plane curves, this implies that $\Delta(P)$ is Whitney equisingular (see \cite{Z}). Therefore, $(\mathcal X,0)$ is Zariski equisingular. 
\end{proof}

We remark that the Whitney equisingularity of the family $(\mathcal X,0)$ does not imply the constancy of $c(X_t,0)$. In fact, the following example, due to Brian\c con and Speder in \cite{BS}, shows a Whitney equisingular family with non constant number of cusps.

\begin{ex}
Let $(\mathcal X,0)$ be the surface in $\C^4$ defined by the zero set of $\Phi(x,y,z,t)=x^6+y^6+z^3+tx^4z$. We write $X_t=\phi_t^{-1}(0)$, where $\phi_t(x,y,z)=\Phi(x,y,z,t)$. Computing $c(X_t,0)$ with the formula of Proposition \ref{cusps}, we get $c(X_0,0)=30$ and $c(X_t,0)=24$ for $t\neq0$.
\end{ex}

\begin{remark}
By Theorem \ref{54}, we know that if a family $f_t:(X_t,0)\to(\C^2,0)$ is topologically trivial, then $\mu(\Delta_t)$ is constant. A natural open question is whether the converse is true. 

We do not know the answer in general, but we show that it is true for a family of generic linear projections $p_t:(X_t,0)\to(\C^2,0)$. In fact, if $\mu(\Delta(p_t))$ is constant, then we know that $p_t^{-1}(\Delta_t)$ is Whitney equisingular by \cite[Corollary 2.11]{NT} and therefore, $m_0(p_t^{-1}(\Delta_t))$ is constant. Moreover, since $\Delta_t$ are plane curves, $\Delta_t$ is Whitney equisingular and hence, by the Speder's result, the family $X_t$ is Whitney equisingular. Then, the family $p_t$ is Whitney equisingular by Theorem \ref{we}. 
\end{remark}

\section{The weighted homogeneous case}

In this section, we will show how to compute the invariants of a weighted homogeneous germ on a weighted homogenous ICIS in terms of the weights and degrees. 

We fix positive integer numbers $w_1,\dots,w_n,d$. We say that a function germ $h\in\mathcal O_n$ is weighted homogeneous of  type $(w_1,\dots,w_n;d)$ if
\begin{equation*}
h(t^{w_1}x_1,\dots,t^{w_n}x_n)=t^{d}h(x_1,\dots,x_n),
\end{equation*}
for all $t\in \C$ and $(x_1,\dots,x_n)\in\C^n$. In this case, we call $(w_1,\dots,w_n)$ the weights of $h$ and $d$, the weighted degree of $h$. We say that a map germ $H=(h_1,\dots,h_p):(\C^n,0)\to\C^p$ is weighted homogeneous of type $(w_1,\dots,w_n;d_1,\dots,d_p)$ if each $h_i$ is a weighted homogeneous function germ of type $(w_1,\dots,w_n;d_i)$.
If $(X,0)$ is an ICIS defined as the zeros of a weighted homogeneous map germ of type $(w_1,\dots,w_n;d_1,\dots,d_p)$, then we say that $(X,0)$ is weighted homogeneous of type $(w_1,\dots,w_n;d_1,\dots,d_p)$.

In the following proposition, we assume that $(X,0)\subset(\C^n,0)$ is a weighted homogeneous ICIS of dimension $2$ of type $(w_1,\dots,w_n;d_3,\dots,d_{n})$ and that $f:(X,0)\to(\C^2,0)$ is an $\mathscr A$-finite weighted homogeneous map germ of type $(w_1,\dots,w_n;d_1,d_2)$. 
In order to simplify the notation, we will write
\begin{align*}
&A=\sum_{i=1}^nw_i, \, \, A_2=\sum_{1\leq i<j\leq n}w_iw_j, \, \, B=\prod_{i=1}^nw_i,\\
&C=\sum_{i=1}^nd_i, \, \, C_2=\sum_{3\leq i\leq j\leq n}d_id_j, \,\, D=\prod_{i=1}^n d_i.
\end{align*}

\begin{prop}\label{musw} With the above notation, we have:
\begin{align*}
\mu(S)&=\frac{D(C-A)[2(C-A)-d_{1}-d_2]}{Bd_{1}d_2}+1,\\
c&=\frac{D}{Bd_{1}d_2}\left[2(C-A)^2-C(d_{1}+d_2-A)+d_{1}d_2-A_2-C_2\right].
\end{align*}															
\end{prop}

\begin{proof}
If $X=\phi^{-1}(0)$ with $\phi=(\phi_1,\dots,\phi_{n-2})$ and $f=(f_1,f_2)$, then $S$ is the curve
$$v(\phi_1,\dots,\phi_{n-2},\det J(\phi_1,\dots,\phi_{n-2},f_1,f_2)).$$
Therefore, $S$ is a weighted homogeneous ICIS of type 
$$(w_1,\dots,w_n;d_3,\dots,d_{n},C-A).$$
The result for $\mu(S)$ follows now from \cite[Corollary 4.2]{NOT}.

The multiplicity $m$ can be computed easily by using Bezout's theorem:
\begin{equation*}
m=\dim_{\C}\frac{\mathcal O_n}{\left\langle \phi_1,\dots,\phi_{n-2},f_1,f_2\right\rangle}=\frac{D}{B}.
\end{equation*}

Finally, from \cite[theorem 7.9]{GR}, we have that
\begin{equation*}
\mu(X,0)=-1+\frac{D}{d_{1}d_2B}(A_2+C_2-A(C-d_{1}-d_2)).
\end{equation*}
The result for $c$ follows by using Theorem \ref{2} and combining the formulas for $\mu(S)$, $m$ and $\mu(X,0)$.
\end{proof}

%\begin{prop}\label{mw} Let $(X,0)\subset(\C^n,0)$ be an ICIS of dimension $2$ defined by a weighted homogeneous map germ of  type $(w_1,\dots,w_n;d_3,\dots,d_{n})$ and let $f:(X,0)\to(\C^2,0)$ be an $\mathscr A$-finite weighted homogeneous map of type $(w_1,\dots,w_n;d_{1},d_2)$. Then,
%\begin{equation*}
%m=\frac{D}{B}.
%\end{equation*}
%\end{prop}
%
%\begin{proof}
%We have that
%\begin{equation*}
%m=\dim_{\C}\frac{\mathcal O_n}{\left\langle \phi_1,\dots,\phi_{n-2},f_1,f_2\right\rangle}.
%\end{equation*}
%Therefore, the result is a consequence of Bezout's theorem.
%
%Finally, by \cite[theorem 7.9]{GR}, we have that
%\begin{equation*}
%\mu(X,0)=-1+\frac{D}{d_{1}d_2B}(A_2+C_2-A(C-d_{1}-d_2)).
%\end{equation*}
%The result follows if we combine this formula with the propositions \ref{musw}, \ref{mw} and the theorem \ref{2}. 
%
%\end{proof} 
%
%\begin{cor} Let $(X,0)\subset(\C^n,0)$ be an ICIS with dimension $2$ defined by a weighted homogeneous map germ of the type $(w_1,\dots,w_n;d_3,\dots,d_{n})$ and let $f:(X,0)\to(\C^2,0)$ be an $\mathcal A$-finite weighted homogeneous map of type $(w_1,\dots,w_n;d_{1},d_2)$. Then,
%\begin{equation*}
%c=\frac{D}{Bd_{1}d_2}\left[2(C-A)^2-C(d_{1}+d_2-A)+d_{1}d_2-A_2-C_2\right].
%\end{equation*}
%\end{cor}
%
%\begin{proof}
%By the \cite[theorem 7.9]{GR}, we have that
%\begin{equation*}
%\mu(X,0)=-1+\frac{D}{d_{1}d_2B}(A_2+C_2-A(C-d_{1}-d_2)).
%\end{equation*}
%
%The result follows if we combine this formula with the propositions \ref{musw}, \ref{mw} and the theorem \ref{2}. 
%\end{proof}

We will need the following lemma in order to write $\mu(\Delta)$ and $d$ in terms of the weights and degrees.

\begin{lem}
Let $S\subset (\C^n,0)$ be a weighted homogenous ICIS of dimension $1$ of type $(w_1,\dots,w_n;c_1,\dots,c_{n-1})$ and let $f=(f_1,f_2):(\C^n,0)\to\C^2$ be a weighted homogeneous map germ of  type $(w_1,\dots,w_n;c_n,c_{n+1})$ such that $f|_S:S\to\C^2$ is $s$-to-one. Then $\Delta:=f(S)$ is weighted homogeneous of type $(c_n,c_{n+1};\frac{c_1\dots c_{n+1}}{sw_1\dots w_n})$.
\end{lem}

\begin{proof}

{ \it Case 1:} The homogeneous case. 

We assume that $w_1=\dots=w_n=1$. Let $S_i$ be a branch of $S$. Since $S$ is homogenous, the map
\begin{equation*}
\gamma_i(t)=(a_1t,\dots,a_nt)
\end{equation*}
is a parametrization of $S_i$, where $0\neq(a_1,\dots,a_n)\in S_i$.
We have
\begin{equation*}
f(\gamma(t))=(t^{c_n}f(a_1,\dots,a_n), t^{c_{n+1}}f(a_1,\dots,a_n)).
\end{equation*}
Then, each branch of $\Delta$ is weighted homogeneous of type $(c_n,c_{n+1};c_nc_{n+1})$ and $\Delta$ is weighted homogeneous of type $(c_n,c_{n+1};rc_nc_{n+1})$, where $r$ is the number of branches of $\Delta$.

Since $f|_S$ is $s$-to-one, the number of branches of $S$ is equal to $sr$. On the other hand, since $S$ is homogeneous with degrees $c_1,\dots,c_{n-1}$, by Bezout's theorem, the number of branches of $S$ is $c_1\dots c_{n-1}$. Thus, we deduce that $r=\frac{c_1\dots c_{n-1}}{s}$.

{\it Case 2:} The general case. 

Assume that $S=\phi^{-1}(0)$, where $\phi$ is weighted homogeneous of type $(w_1,\dots,w_n;c_1,\dots,c_{n-1})$.
Let $\alpha:\C^n\to\C^n$ be the map defined by 
$$\alpha(x_1,\dots,x_n)=(x_1^{w_1},\dots,x_n^{w_n})$$ 
and let $\tilde{S}$ be the curve defined by $\tilde S=(\phi\circ\alpha)^{-1}(0)$. 

We have that $\phi\circ\alpha$ is homogeneous with degrees $c_1,\dots,c_{n-1}$ and that $(f\circ\alpha)|_{\tilde{S}}$ is homogeneous with degrees $c_n,c_{n+1}$ and is $(sw_1\dots w_n)$-to-one. By applying the Case 1 to $(f\circ\alpha)|_{\tilde{S}}$, we conclude that $\Delta=f(S)=f\circ\alpha(\tilde{S})$ is weighted homogeneous of type $(c_n,c_{n+1};\frac{c_1\dots c_{n+1}}{sw_1\dots w_n})$.
\end{proof}

Now, we can write $\mu(\Delta)$ , $d$ and $\mu_\Delta$ as a function of the weights and degrees.

\begin{prop}
With the same notation as in Proposition \ref{musw}, we have:
\begin{equation*}
\mu(\Delta)=\frac{\left[D(C-A)-d_{1}B\right]\left[D(C-A)-d_2B\right]}{d_{1}d_2B}.
\end{equation*}															
\end{prop}

\begin{proof}
In fact, we know from the proof of Proposition \ref{musw} that $S$ is weighted homogeneous of type $(w_1,\dots,w_n:d_3,\dots,d_{n},C-A)$. Therefore, by the previous lemma $\Delta=f(S)$ is weighted homogeneous of type $(d_{1},d_2;\frac{D(C-A)}{B})$. Therefore, the result follows from Milnor-Orlik formula.
\end{proof}

We recall that, by Theorem \ref{1},
\begin{equation*}
d=\frac{\mu(\Delta)}{2}-\frac{\mu(S)}{2}-c,
\end{equation*}
and, by Proposition \ref{mudelta},
\begin{equation*}
\mu_\Delta(f)=d+\mu(S).
\end{equation*}
So, it is not hard to compute $d$ and $\mu_\Delta(f)$ in terms of the weights and degrees.

\end{document}